\documentclass[reqno]{amsart}
\usepackage[utf8]{inputenc}
\usepackage[a4paper]{geometry}
\usepackage[english]{babel}
\usepackage{amsfonts, amsthm, amssymb, amsmath}
\usepackage{hyperref}
\usepackage{xcolor}
\usepackage{cleveref}
\usepackage{enumitem}

\theoremstyle{plain}
\newtheorem{theo}{Theorem}[section]
\newtheorem{prop}[theo]{Proposition}
\newtheorem{lemma}[theo]{Lemma}
\newtheorem{coro}[theo]{Corollary}

\newtheorem{question}[theo]{Question}
\theoremstyle{definition}
\newtheorem{rem}[theo]{Remark}
\newtheorem{example}[theo]{Example}

\newcommand{\C}{\mathbb{C}}
\newcommand{\D}{\mathbb{D}}

\title[Some properties of completely convex functions]{\textbf{Some properties of completely convex functions} \\ \small \today}
\date{}

\author[H. Arbeláez]{Hugo Arbeláez}
\address{Hugo Arbeláez, Escuela de Matemáticas, Universidad Nacional de Colombia, Medellín, Colombia}
\email{hjarbela@unal.edu.co}

\author[R. Hernández]{Rodrigo Hernández}
\address{Rodrigo Hernández, Facultad de Ingeniería y Ciencias, Universidad Adolfo Ibáñez, Av. Padre Hurtado 750, Viña del Mar, Chile}
\email{rodrigo.hernandez@uai.cl}

\thanks{The first author was supported by the Universidad Nacional de Colombia code Hermes XXXX.}

\newcommand{\pend}[1]{\par\noindent\textcolor{red}{\textbf{[PENDIENTE]} #1}\par}

\begin{document}

\maketitle

\begin{abstract}

We study the class $CC_\alpha$ of completely convex functions of order $\alpha$, defined through a two-point analogous of the classical convexity condition. We give several equivalent characterizations of $CC_\alpha$, together with representation formulas in terms of two auxiliary Schwarz functions linked by a compatibility equation. From these representations, we derive coefficient bounds, a subordination result $zf'/f$, and a two-point Taylor expansion in which the Schwarzian derivative $S_f$ arises naturally as the coefficient of the linear term. This yields a pointwise bound for $S_f$ and a bound for the pre-Schwarzian norm $\|P_f\|_{CC_\alpha}$, improving the classical Yamashita bound when $\alpha\leq 1-1/\sqrt2 $. We also show that, unlike the class $C_\alpha$ of convex functions of order $\alpha$, every $f\in CC_\alpha$ is bounded when $\alpha>0$.
\end{abstract}

\vspace{0.3cm}


\section{Introduction}

Let $\D=\{z\in\C:|z|<1\}$ be the unit disk. We denote by $\mathcal{A}$ the class of analytic functions $f : \D \to \C$ normalized by $f(0)=0$ and $f'(0)=1$. The classes of convex, starlike, convex of order $\alpha$, and starlike of order $\alpha$ functions are denoted $C$, $S^*$, $C_\alpha$, and $S^*_\alpha$, respectively, with $\alpha \in [0,1)$.
 
A domain is convex if and only if it is starlike with respect to each of its points. Suffridge \cite{S1970} expressed this in terms of the mapping itself: a function $f\in\mathcal A$ is convex if and only if
\[
  \operatorname{Re}\left\{\frac{2zf'(z)}{f(z)-f(x)}-\frac{z+x}{z-x}\right\}\geq0,\qquad x,z\in\D.
\]
For $x=z$ the expression reduces to $1+zf''(z)/f'(z)$, so the usual definition of convexity is the diagonal case of this two-point condition, and for convex functions nothing more is required. Convexity of order $\alpha$ is obtained by relaxing the one-point condition to $\operatorname{Re}\{1+zf''(z)/f'(z)\}>\alpha$ (see \cite{Rob1936}). Replacing $0$ by $\alpha$ in the two-point condition instead gives the class $CC_\alpha$ of \emph{completely convex} functions of order $\alpha$: those $f\in\mathcal A$ for which the real part above is at least $\alpha$ for all $x,z\in\D$ (see \eqref{eq:CCa}). Setting $x=z$ shows that $CC_\alpha\subseteq C_\alpha$, with equality for $\alpha=0$, and setting $x=0$ gives $CC_\alpha\subseteq S^*_{(1+\alpha)/2}$. For $\alpha>0$ the two ways of introducing the order $\alpha$ no longer agree: functions in $CC_\alpha$ are bounded, while $C_\alpha$ contains unbounded functions when $\alpha\leq1/2$. The Möbius maps $z/(1-\lambda z)$ belong to $CC_\alpha$ exactly when $|\lambda|\leq(1-\alpha)/(1+\alpha)$.

We are interested in what these functions look like. The two-point condition is equivalent to the existence of an analytic function $\delta:\D\times\D\to\D$ that controls the difference quotients of $f$, and it restricts $f$ more than its one-point consequences suggest. We use it to study the coefficients, the pre-Schwarzian and Schwarzian derivatives, and the growth and distortion of $f$ and $f'$ through subordination. The arguments rely only on classical tools (the Schwarz--Pick lemma, the Carath\'eodory inequality, the Herglotz representation, subordination \cite{MM2000}), and that they suffice says something about how much structure the kernel carries.

Section~2 starts from the description of $CC_\alpha$ through $\delta$, taken from \cite{paper}, and derives two representations of $f$ by Schwarz functions, a compatibility relation between them, and the starlikeness of order $\alpha$ of the odd part of $f$. In Section~3 the Schwarzian derivative appears as the coefficient of the linear term in a two-point expansion of the kernel. Using the derivatives of the kernel in both variables we obtain a pointwise inequality relating $P_f$ and $S_f$, and a bound for the pre-Schwarzian norm that improves Yamashita's estimate \cite{Y1999} when $\alpha\leq1-1/\sqrt2$. Section~4 collects coefficient estimates, among them $|a_2|\leq(1-\alpha)\sqrt{2/(2+\alpha)}$ and $|a_3-a_2^2|\leq(1-\alpha-|a_2|^2)/3$, which reduces to Trimble's inequality \cite{T1975} when $\alpha=0$. In Section~5 we show that the difference quotient of $f$ is subordinate to $(1-z)^{-(1-\alpha)}$, which gives the subordination of $zf'/f$ and two-sided bounds for $|f'|$.

\section{The class of completely convex functions}

The class of completely convex functions $CC_\alpha$, $\alpha \in [0,1)$, consists of functions
$f \in \mathcal{A}$ that satisfy
\begin{equation}\label{eq:CCa}
  \operatorname{Re}\!\left\{
    \frac{2zf'(z)}{f(z)-f(x)} - \frac{z+x}{z-x}
  \right\} \geq \alpha, \qquad x, z \in \D.
\end{equation}
Setting $x=z$ gives $CC_\alpha \subseteq C_\alpha$, and setting $x=0$ gives $CC_\alpha \subseteq S^*_{(1+\alpha)/2}$. Suffridge in \cite{S1970} shows that $CC_0 = C$.

The condition~\eqref{eq:CCa} is equivalent (see \cite[Rem.~2]{paper})
to the representation
\begin{equation}\label{eq:delta}
  \frac{f'(z)}{f(z)-f(x)}-\frac{1}{z-x}=\frac{(1-\alpha)\,\delta(x,z)}{1-z\,\delta(x,z)},
\end{equation}
where $\delta:\D \times \D \to \D$ is analytic and $w(x,z):=z\,\delta(x,z)$ satisfies $w(x,0)=0$ for every $x \in \D$. The representation \eqref{eq:delta} yields a quantitative
reformulation of $CC_\alpha$ that makes the role of $|\delta|\leq 1$ completely explicit.

\begin{theo}\label{prop:quant_char}
  Let $f\in\mathcal{A}$ be locally univalent and
\begin{equation}\label{eq:A}
  A = A(x,z) \;:=\; \frac{f'(z)}{f(z)-f(x)} -\frac{1}{z-x}.
\end{equation}
For $\alpha\in [0,1)$, the following statements are equivalent
\begin{enumerate}
\item [(i)] $f\in CC_\alpha$
\item [(ii)] 
\begin{equation}\label{eq:quant_char}
\frac{(1-|z|^2)\,|A|^2}{1-\alpha}+\alpha\leq
  \operatorname{Re}\left\{\frac{2z f'(z)}{f(z)-f(x)}-\frac{z+x}{z-x}\right\},\quad  x,z\in\D
\end{equation}
\item [(iii)] $\left|(1-|z|^2)A-(1-\alpha) \overline{z}\right|\leq 1-\alpha$, for all  $x,z\in \D.$
\end{enumerate}  
\end{theo}

\begin{proof}
By \eqref{eq:delta} we have $A=(1-\alpha)\delta/(1-z\delta)$, which implies $\delta=A/(1-\alpha+Az)$.
Since $|\delta|\leq 1$, then
\[
  |A|^2\leq|1-\alpha+Az|^2=(1-\alpha)^2+2(1-\alpha)\operatorname{Re}\{Az\}+|A|^2|z|^2,
\]
equivalently
\begin{equation}\label{eq:key_ineq}
  \frac{|A|^2(1-|z|^2)}{1-\alpha}+\alpha
  \leq 1+2\operatorname{Re}\{Az\}.
\end{equation}
Now,
\[
 2zA=\frac{2zf'(z)}{f(z)-f(x)}-\frac{2z}{z-x} =\frac{2zf'(z)}{f(z)-f(x)}-\frac{z+x}{z-x}-1, 
\]
substituting into \eqref{eq:key_ineq}, it follows
\[
  \frac{|A|^2(1-|z|^2)}{1-\alpha}+\alpha
  \leq \operatorname{Re}\left\{\frac{2zf'(z)}{f(z)-f(x)}-\frac{z+x}{z-x}\right\}, 
\]
which proves that (i) is equivalent to (ii). Now, by \eqref{eq:key_ineq}, multiplying by $1-|z|^2$
\begin{equation*}
|A|^2(1-|z|^2)^2-2(1-\alpha)(1-|z|^2) \operatorname{Re}\{Az\}+(1-\alpha)^2|z|^2 \leq(1-\alpha)^2,  
\end{equation*}
equivalently
\begin{equation*}
\left |(1-|z|^2)A-(1-\alpha)\overline{z} \right|^2\leq (1-\alpha)^2.
\end{equation*}
\end{proof}

\begin{rem}
The inequality \eqref{eq:quant_char} expresses
this as a lower bound on the two-point kernel that is strictly larger than $\alpha$ whenever $|z|<1$ and $A\neq 0$. For $\alpha=0$, the condition \eqref{eq:quant_char} reads
\[
  (1-|z|^2)\left|\frac{f'(z)}{f(z)-f(x)}-\frac{1}{z-x}\right|^2\leq\operatorname{Re}\left\{
\frac{2zf'(z)}{f(z)-f(x)}-\frac{z+x}{z-x}
 \right\},
\]
the classical condition of Suffridge for convex functions in a self-improving form.
\end{rem}

\begin{prop}\label{prop:repz0}
Every $f \in CC_\alpha$ admits the representation
\begin{equation}\label{eq:repz0}f(z)=\frac{z}{1-(1-\alpha)\varphi(z)},
\end{equation}
where $\varphi:\D \to \D$ is analytic with $\varphi(0)=0$. In particular, if $\alpha>0$ then  $|f(z)|<1/\alpha$ for all $z\in\D$.
\end{prop}

\begin{proof}
Evaluate \eqref{eq:delta} at $z = 0$ we obtain
\[
  \frac{1}{x}-\frac{1}{f(x)}=(1-\alpha)\delta(x,0).
\]
Setting $\varphi(x) := x\,\delta(x,0)$ and solving for $f(x)$, it follows \eqref{eq:repz0}. The last statement is obtained by directly applying the triangle inequality.
\end{proof}

\begin{rem}\label{rem:repz0}
The boundedness given by Proposition \ref{prop:repz0} is a feature that sharply distinguishes $CC_\alpha$ from $C_\alpha$ for $\alpha > 0$.  Also, the representation is \emph{necessary} but not sufficient, namely, not every Schwarz function $\varphi$ produces a function in $CC_\alpha$. For instance, $\varphi(z) = z$ gives $f(z) = z/(1-(1-\alpha)z)$, which does \emph{not} belong to $CC_\alpha$ when $\alpha > 0$ (see Example \ref{ex:mobius} below). 
\end{rem}

\begin{example}\label{ex:mobius}
Let $T_\lambda(z) = z/(1-\lambda z)$ with $\lambda \in \C$. A direct computation gives $T_\lambda(z)-T_\lambda(x)=(z-x)/[(1-\lambda z)(1-\lambda x)]$, whence the two-point kernel collapses to the one-point kernel:
\[
  \frac{2zT_\lambda'(z)}{T_\lambda(z)-T_\lambda(x)}-\frac{z+x}{z-x}=\frac{1+\lambda z}{1-\lambda z},
\]
independently of $x$. Since the minimum of
$\operatorname{Re}\{(1+\lambda z)/(1-\lambda z)\}$ over $\D$ equals $(1-|\lambda|)/(1+|\lambda|)$, we obtain
\begin{equation*}
 T_\lambda \in CC_\alpha \iff |\lambda|\leq \frac{1-\alpha}{1+\alpha}.
\end{equation*}
The associated Schwarz function of Proposition \ref{prop:repz0} is
$\varphi(z)=\lambda z/(1-\alpha)$, so that for the extremal choice
$|\lambda|=(1-\alpha)/(1+\alpha)$, $\alpha > 0$,
gives 
\begin{equation*}
  \sup_{z \in \D}|\varphi(z)|=\frac{1}{1+\alpha}.
\end{equation*}
\end{example}

\begin{rem}\label{rem:bound_mobius}
For the extremal maps $T_\lambda$, $|\lambda|=(1-\alpha)/(1+\alpha)$, one has
$\sup_{z\in\D}|T_\lambda(z)|=1/(1-|\lambda|)=(1+\alpha)/(2\alpha)$, which is strictly smaller than the bound $1/\alpha$ of Proposition \ref{prop:repz0}. We do not know the sharp bound for $|f|$ in $CC_\alpha$; by Remark~\ref{rem:fc}, the maps $T_\lambda$ are not extremal for it.
\end{rem}

\begin{prop}\label{prop:repx0}
Let $f \in CC_\alpha$ be, then
\begin{equation}\label{eq:repx0}
\frac{zf'}{f}(z)=\frac{1-\alpha\psi(z)}{1-\psi(z)},
\end{equation}
where $\psi:\D \to \D$ is analytic with $\psi(0)= 0$.
\end{prop}

\begin{proof}
Evaluate \eqref{eq:delta} at $x = 0$ we have
\[
  \frac{f'}{f}(z)-\frac{1}{z}=\frac{(1-\alpha) \delta(0,z)}{1-z\,\delta(0,z)}.
\]
Setting $\psi(z) := z\,\delta(0,z)$ we obtain \eqref{eq:repx0}.
\end{proof}

\begin{rem}
Equation~\eqref{eq:repx0} is the Herglotz  representation of $f \in S^*_{(1+\alpha)/2}$, consistent with
$CC_\alpha \subset S^*_{(1+\alpha)/2}$.
\end{rem}

\begin{prop}\label{prop:compat}
The functions $\varphi$ and $\psi$ in
Propositions \ref{prop:repz0} and \ref{prop:repx0} satisfy the compatibility relation
\begin{equation}\label{eq:compat}
  \frac{\varphi'(z)}{1-(1-\alpha)\varphi(z)} =\frac{\psi(z)}{z\,(1-\psi(z))}.
\end{equation}
In particular, $\psi$ determines $f$ uniquely via integration of \eqref{eq:repx0}, and $\varphi$ is then determined by \eqref{eq:compat}.
\end{prop}

\begin{proof}
Using \eqref{eq:repz0} to differentiate $\log(f(z)/z)$, we obtain $(1-\alpha) \varphi'(z) /(1-(1-\alpha)\varphi(z))$. Also, using \eqref{eq:repx0} we have $f'(z)/f(z)-1/z=(1-\alpha) \psi(z)/(z(1-\psi(z))$.
\end{proof}

\begin{theo}\label{theo:oddpart}
Let $f\in CC_\alpha$, $\alpha \in [0,1)$, and let
\begin{equation}\label{eq:oddpart}
  F(z) \;:=\; \frac{f(z)-f(-z)}{2}
\end{equation}
be the odd part of $f$. Then
\begin{equation*}
  \operatorname{Re}\left\{\frac{zF'}{F}(z)\right\}\geq \alpha+\frac{1}{2}\frac{1-|z|^2}{1-\alpha}\left(|A(z,-z)|^2+|A(-z,z)|^2\right), 
\end{equation*}
for all $z\in \D$, where $A$ is defined by \eqref{eq:A}. In particular, $F\in S^*_\alpha$.
\end{theo}

\begin{proof}
By \eqref{eq:quant_char}, with the roles
of $z$ and $x$ exchanged
\begin{equation}\label{prop:symkernel1}
  \operatorname{Re}\left\{\frac{2xf'(x)}{f(x)-f(z)}-\frac{x+z}{x-z}\right\} \geq \alpha +\frac{(1-|x|^2)|A(z,x)|^2}{1-\alpha}.
\end{equation}
Adding \eqref{eq:quant_char} and \eqref{prop:symkernel1}, we obtain
\begin{equation}\label{eq:symkernel}
 \operatorname{Re}\left\{
 \frac{zf'(z)-xf'(x)}{f(z)-f(x)}
  \right\}\geq\alpha+\frac{1}{2(1-\alpha)}\left((1-|x|^2)|A(z,x)|^2+(1-|z|^2)|A(x,z)|^2\right),
\end{equation}
 for all $x,z\in\D$. On the other hand, differentiating \eqref{eq:oddpart}
\[
 F'(z)=\frac{f'(z)+f'(-z)}{2}.
\]
So
\[
  \frac{zF'}{F}(z)=\frac{z\bigl(f'(z)+f'(-z)\bigr)}{f(z)-f(-z)}.
\]
Now, setting $x=-z$ in~\eqref{eq:symkernel},
it follows the result.
\end{proof}

\begin{rem}
For $\alpha=0$ this recovers a classical fact about convex functions:
if $f\in C$ then its odd part is starlike. We observe that for the Möbius maps
of Example~\ref{ex:mobius}, $T_\lambda\in CC_\alpha$ with
$\lambda=(1-\alpha)/(1+\alpha)$, the odd part
$F(z)=z/(1-\lambda^2z^2)$ satisfies
\[
  \inf_{z\in\D}\operatorname{Re}\left\{\frac{zF'}{F}(z)\right\}=\frac{1-\lambda^2}{1+\lambda^2}  =\frac{2\alpha}{1+\alpha^2}>\alpha \qquad (0<\alpha<1),
\]
so the order of starlikeness of $F$ along this family is strictly larger than $\alpha$.
\end{rem}

For what follows, we define
\begin{equation}\label{G}
 G_f(x,z):=\frac{2zf'(z)}{f(z)-f(x)}-\frac{z+x}{z-x}, \quad x,z\in \D.   
\end{equation}

\begin{prop}\label{prop:noincl}
For every $\alpha \in (0, 1/2]$, $C_\alpha \not\subseteq CC_\beta$
for any $\beta > 0$.
\end{prop}

\begin{proof}
Since $C_{1/2} \subseteq C_\alpha$ for $\alpha \in (0,1/2]$, it suffices to exhibit a function in $C_{1/2}$ outside every $CC_\beta$ with $\beta > 0$. Consider $k(z)=-\log(1-z)$. It is easy to note that $1+zk''(z)/k'(z)=1/(1-z)$, so $\operatorname{Re}\{1+zk''(z)/k'(z)\} > 1/2$
and $k \in C_{1/2}$. Now, by Proposition \ref{prop:repz0}, $k \notin CC_\beta$ for any $\beta > 0$, since $k$ is unbounded.
\end{proof}

\begin{rem}\label{rem:beta_question}
The argument in the previous proof breaks down for $\alpha\in(1/2,1)$, where every $f\in C_\alpha$ is bounded, and we do not know whether in that range $C_\alpha\subseteq CC_\beta$ for some $\beta=\beta(\alpha)>0$.
\end{rem}

\section{Pre-Schwarzian and Schwarzian of $f\in CC_\alpha$}

The main result of this section is a upper bound for the pre-Schwarzian norm of functions in $CC_\alpha$, this result improves the bound $\|P_f\|_{C_\alpha} \leq 4(1-\alpha)$, obtained by Yamashita in \cite{Y1999}, for $\alpha \leq 1-1/\sqrt{2}$.

The \emph{pre-Schwarzian derivative} and \emph{Schwarzian derivative}
of a locally univalent $f \in \mathcal{A}$ are defined by
\[
  P_f=\frac{f''}{f'}, \qquad
  S_f=\left(\frac{f''}{f'}\right)'-\frac{1}{2}\left(\frac{f''}{f'}\right)^2,
\]
with norms $\|P_f\|=\sup_{z\in\D}(1-|z|^2)|P_f(z)|$ and
$\|S_f\| = \sup_{z \in \D}(1-|z|^2)^2|S_f(z)|$.

\begin{rem}
Let $f \in CC_\alpha$ be, if we set $t=x-z$,
\begin{equation}\label{eq:taylor}
  G_f(z+t, z)=1 + zP_f(z)+\frac{z}{3}\,S_f(z)\cdot t+ O(t^2).
\end{equation}
Thus, the expressions for the pre-Schwarzian and Schwarzian derivatives appear explicitly in a series expansion for the function $G_f$.
\end{rem}

\begin{theo}\label{theo:Sf}
Let $f\in CC_\alpha$, $\alpha\in [0,1)$. Then, for all $z\in\D$,
\begin{equation}\label{eq:Sfbound}
 (1-|z|^2)\left(\left|\frac z3S_f(z)\right|+\left|P_f(z)+\frac z2P_f(z)^2+\frac{2z}3S_f(z)\right|\right)\leq 2\operatorname{Re}\left\{1-\alpha+zP_f(z)\right\}.
\end{equation}
\end{theo}

\begin{proof}
Write
\[
  a(z):=\frac z3\,S_f(z),\qquad
  b(z):=P_f(z)+\frac z2\,P_f(z)^2+\frac{2z}{3}\,S_f(z),
\]
and $R(z):=\operatorname{Re}\{1-\alpha+zP_f(z)\}$.

\medskip
\noindent\emph{Step 1: the partial derivatives of $G_f$ on the diagonal.}
The singularity of $G_f(x,z)$ at $x=z$ is removable, so $G_f$ is analytic in $\D\times\D$, and $G_f(z,z)=1+zP_f(z)$. By \eqref{eq:taylor}, with $x=z+t$, the coefficient of $t$ is the derivative of $G_f$ with respect to its first variable on the diagonal, that is, $\partial_xG_f(z,z)=a(z)$. On the other hand, differentiating the identity $G_f(z,z)=1+zP_f(z)$ along the diagonal yields $\partial_xG_f(z,z)+\partial_zG_f(z,z)=(1+zP_f)'(z)$. Using $P_f'=S_f+\tfrac12P_f^2$,
\[
  \partial_zG_f(z,z)=P_f(z)+zP_f'(z)-\frac z3S_f(z)=P_f(z)+\frac z2P_f^2(z)+\frac{2z}3S_f(z)=b(z).
\]

\medskip
\noindent\emph{Step 2: a Carath\'eodory argument on the bidisk.}
Fix $z_0\in\D$ and $\mu_1,\mu_2\in\C$ with $|\mu_1|=|\mu_2|=1$, and let $\gamma(w)=(z_0+w)/(1+\bar z_0w)$. The map $t\mapsto(\gamma(\mu_1t),\gamma(\mu_2t))$ sends $\D$ into $\D\times\D$, hence
\[
  q(t)=\frac{G_f\bigl(\gamma(\mu_1t),\gamma(\mu_2t)\bigr)-\alpha}{1-\alpha}
\]
is analytic in $\D$ and $\operatorname{Re}q\ge0$ by \eqref{eq:CCa} (the points on the diagonal are covered by the case $x=z$). Since $\gamma(0)=z_0$ and $\gamma'(0)=1-|z_0|^2$,
\[
  q(0)=\frac{1-\alpha+z_0P_f(z_0)}{1-\alpha},\qquad
  q'(0)=\frac{(1-|z_0|^2)\bigl(\mu_1a(z_0)+\mu_2b(z_0)\bigr)}{1-\alpha}.
\]
The Carath\'eodory inequality $|q'(0)|\le2\operatorname{Re}q(0)$ gives
\[
  (1-|z_0|^2)\,\bigl|\mu_1a(z_0)+\mu_2b(z_0)\bigr|\le2R(z_0).
\]

\medskip
\noindent\emph{Step 3: choice of the rotations.}
Choosing $\mu_1,\mu_2$ so that $\mu_1a(z_0)$ and $\mu_2b(z_0)$ have the same argument, $|\mu_1a(z_0)+\mu_2b(z_0)|=|a(z_0)|+|b(z_0)|$, and therefore
\[
  (1-|z_0|^2)\bigl(|a(z_0)|+|b(z_0)|\bigr)\le2\operatorname{Re}\{1-\alpha+z_0P_f(z_0)\}.
\]
Since $z_0\in\D$ is arbitrary, the proof is complete.
\end{proof}

\begin{rem}
The inequality \eqref{eq:Sfbound} contains the following estimates as particular cases:
\begin{enumerate}
    \item [(i)] \[
  \frac{|z|(1-|z|^2)}{6}\,|S_f(z)|\le\operatorname{Re}\{1+zP_f(z)\}-\alpha
\]
\item [(ii)]
\[
  \Bigl|P_f(z)+\frac z2P_f(z)^2+\frac{2z}3S_f(z)\Bigr|(1-|z|^2)\le2\operatorname{Re}\{1-\alpha+zP_f(z)\}.
\]
\end{enumerate}
Moreover, since $a(z)+b(z)=(1+zP_f)'(z)$, the triangle inequality gives
\[
  |(1+zP_f)'(z)|\,(1-|z|^2)\le2\operatorname{Re}\{1-\alpha+zP_f(z)\},
\]
which is the one-point inequality of $C_\alpha\supset CC_\alpha$. The point of the joint statement is that it combines the information of both partial derivatives of $G_f$, rather than that of each one separately. As an estimate for $\|S_f\|$ it does not improve the bounds of Nehari and Suita, and we do not claim that it does.
\end{rem}

\begin{lemma}\label{theo:Pf}
Let $f \in CC_\alpha$, $\alpha \in [0,1)$, and associated Schwarz functions $\psi$ and $\varphi$, then
\begin{equation*}
  \frac{f''}{f'}(z)=\frac{\psi'(z)}{1-\psi(z)}-\frac{\alpha \psi'(z)}{1-\alpha \psi(z)}  +\frac{(1-\alpha)\varphi'(z)}{1-(1-\alpha) \varphi(z)},
\end{equation*}
for all $z\in \D$.
\end{lemma}

\begin{proof}
From \eqref{eq:repx0}, differentiating $\log(zf'(z)/f(z))$, we have
\[
  \frac{1}{z}+\frac{f''}{f'}(z)-\frac{f'}{f}(z)  =-\frac{\alpha \psi'(z)}{1-\alpha \psi(z)}+\frac{\psi'(z)}{1-\psi(z)},
\]
for all $z\in \D$.
Now, by \eqref{eq:repx0} and \eqref{eq:compat} 
\[
\frac{f'}{f}(z)-\frac{1}{z}=\frac{(1-\alpha)\psi(z)}{z(1-\psi(z))}=\frac{(1-\alpha)\varphi'(z)}{1-(1-\alpha)\varphi(z)}.
\]
\end{proof}

\begin{theo}\label{theo:Pf_norm}
Let $f \in CC_\alpha$, $\alpha \in [0,1)$. Then
\begin{equation*}
 \|P_f\|_{CC_\alpha}\leq 2+\frac{2(1-\alpha)}{1+\sqrt{1-(1-\alpha)^2}}.
\end{equation*}
The term on the right-hand side is less than $4(1-\alpha)$ for all $\alpha \leq 1-1/\sqrt{2}$.
\end{theo}

\begin{proof}
From Lemma \ref{theo:Pf} and the Schwarz-Pick lemma
\begin{align}\label{preS}
    (1-|z|^2)\left|\frac{f''}{f'}(z)\right|&
  \leq  \frac{(1-\alpha)(1-|z|^2)|\psi'(z)|}{|1-\psi(z)||1-\alpha \psi(z)|}+\frac{(1-\alpha)(1-|z|^2)|\varphi'(z)|}{|1-(1-\alpha)\varphi(z)|} \nonumber
  \\
 & \leq \frac{(1-\alpha)(1-|\psi(z)|^2)}{|1-\psi(z)||1-\alpha \psi(z)|}+\frac{(1-\alpha)(1-|\varphi(z)|^2)}{|1-(1-\alpha) \varphi(z)|}. 
\end{align}
 We observe that
 \begin{equation}\label{preS1}
  \frac{1-|\psi(z)|^2}{|1-\psi(z)||1-\alpha \psi(z)|}\leq  \frac{1+|\psi(z)|}{1-\alpha|\psi(z)|} \leq \frac{2}{1-\alpha}.
 \end{equation}
 On the other hand, the function
\[
  H(s)=\frac{1-s^2}{1-(1-\alpha)s}, \quad 0\leq s<1,
\]
achieved the maximum value at the point
$s_0 = (1-\sqrt{1-(1-\alpha)^2})/(1-\alpha)$, giving
$H(s_0) = 2/(1+\sqrt{1-(1-\alpha)^2})$. Therefore
\begin{equation}\label{preS2}
 \frac{1-|\varphi(z)|^2}{|1-(1-\alpha) \varphi(z)|} \leq \frac{2}{1+\sqrt{1-(1-\alpha)^2}}.  
\end{equation}
Now, with \eqref{preS1} and \eqref{preS2} in \eqref{preS}, we obtain
\[
  (1-|z|^2)\left|\frac{f''}{f'}(z)\right|
  \leq 2 + \frac{2(1-\alpha)}{1+\sqrt{1-(1-\alpha)^2}}.
\]
The last statement follows from a direct calculation.
\end{proof}

\begin{rem}
For $\alpha = 0$, we recover the Yamashita bound $\|P_f\|_C \leq 4$ for convex functions. For $\alpha \leq 1-1/\sqrt{2}$, the bound is less than $4(1-\alpha)$, confirming that the two-point condition provides an improvement.
\end{rem}

\section{Coefficient bounds}
In this section, we study the coefficients of the functions in the class $CC_\alpha$, seeking to better understand the growth of the pre-Schwarzian and Schwarzian derivatives of $f$.

Inequality \eqref{eq:Sfbound} already contains the first coefficient bound. Let $f(z)=z+a_2z^2+\cdots\in CC_\alpha$. At $z=0$ every term containing the factor $z$ vanishes and $P_f(0)=f''(0)=2a_2$, so \eqref{eq:Sfbound} reads 
\begin{equation}\label{eq:a2_first}
  |a_2|\leq1-\alpha .
\end{equation}
This estimate is improved in Corollary~\ref{coro:a2} below.

We now obtain a family of coefficient inequalities parameterized by $\lambda \in \overline\D$.

\begin{theo}\label{theo:coeff}
Let $f(z)=z+a_2 z^2+a_3 z^3 + \cdots \in CC_\alpha$. Then, for every $\lambda \in \overline\D$,
\begin{equation}\label{eq:coeff_family}
  \bigl|2(1-\alpha)a_3-(2-\alpha)a_2^2+\lambda(1-\alpha)(a_3-a_2^2)\bigr|\leq (1-\alpha)^2-|a_2|^2.
\end{equation}
\end{theo}

\begin{proof}
Let $G_f(x,z)$ be defined by \eqref{G}. A straightforward calculation shows
\begin{equation}\label{G1}
G_f(x,z)=1+2a_2z+2(2a_3-a_2^2)z^2+2(a_3-a_2^2) xz+O(x^3,z^3).
\end{equation}
On the other hand, writing $G_f(x,z)=(1+(1-2 \alpha)w(x,z)/(1-w(x,z))$ with $w(x,z)=z\delta(x,z)$ and expanding $w(x,z)$ in power series and comparing the coefficients with \eqref{G1}, we obtain 
\begin{equation*}
 \delta(x,z)=\frac{a_2}{1-\alpha}+\frac{a_3-a_2^2}{1-\alpha}x+\left(\frac{2a_3-a_2^2}{1-\alpha}-\frac{a_2^2}{(1-\alpha)^2}\right)z-\frac{2a_2(a_3-a_2^2)}{(1-\alpha)^2}xz+ O(x^2,z^2).   
\end{equation*}
Setting $x=\lambda z$, we have
 \[
  \delta(\lambda z,z)=\frac{a_2}{1-\alpha}
  +\left[\frac{2a_3-a_2^2}{1-\alpha}
  -\frac{a_2^2}{(1-\alpha)^2}
  +\lambda\frac{a_3-a_2^2}{1-\alpha}\right] z+ O(z^2).
\]
Since $z \mapsto \delta(\lambda z,z)$ maps $\D$ to $\D$ with value $a_2/(1-\alpha)$ at $z=0$, the Schwarz-Pick lemma gives 
\begin{equation*}
\left|\frac{2a_3-a_2^2}{1-\alpha}-\frac{a_2^2}{(1-\alpha)^2}+\lambda\frac{a_3-a_2^2}{1-\alpha} \right|\leq 1-\left|\frac{a_2}{1-\alpha}\right|^2.
\end{equation*}
Multiplying by $(1-\alpha)^2$ yields \eqref{eq:coeff_family}.
\end{proof}

\begin{coro}\label{coro:a2}
Let $f(z)=z+a_2z^2+\cdots\in CC_\alpha$, $\alpha\in[0,1)$. Then
\begin{equation}\label{eq:a2_improved}
  |a_2|\leq(1-\alpha)\sqrt{\frac{2}{2+\alpha}}.
\end{equation}
In particular, $|a_2|<1-\alpha$ for every $\alpha\in(0,1)$.
\end{coro}

\begin{proof}
Put
\[
  X=2(1-\alpha)a_3-(2-\alpha)a_2^2,\qquad Y=(1-\alpha)(a_3-a_2^2),
\]
so that \eqref{eq:coeff_family} reads $|X+\lambda Y|\leq M$ for all $\lambda\in\overline\D$, where $M:=(1-\alpha)^2-|a_2|^2$. Choosing $|\lambda|=1$ with $\arg(\lambda Y)=\arg X$ gives $|X|+|Y|\leq M$. On the other hand, the terms in $a_3$ cancel in
\[
  X-2Y=-(2-\alpha)a_2^2+2(1-\alpha)a_2^2=-\alpha a_2^2 .
\]
Hence
\[
  \alpha|a_2|^2=|X-2Y|\leq|X|+2|Y|\leq2\bigl(|X|+|Y|\bigr)\leq2\bigl((1-\alpha)^2-|a_2|^2\bigr),
\]
that is, $(2+\alpha)|a_2|^2\leq2(1-\alpha)^2$, which is \eqref{eq:a2_improved}.
\end{proof}

\begin{rem}
For $\alpha=0$ the bound \eqref{eq:a2_improved} reduces to $|a_2|\leq1$, the classical bound for convex functions, and for $\alpha>0$ it improves \eqref{eq:a2_first}. We do not know whether \eqref{eq:a2_improved} is sharp. By Example~\ref{ex:mobius} and Proposition~\ref{prop:fc} below, the sharp constant $A(\alpha)=\sup\{|a_2|: f\in CC_\alpha\}$ satisfies, for $\frac12<\alpha<1$,
\[
  \max\left\{\frac{1-\alpha}{1+\alpha},\ \frac{\alpha(1-\alpha)}{2\bigl(\alpha^2+(1-\alpha)^2\bigr)}\right\}\leq A(\alpha)\leq(1-\alpha)\sqrt{\frac{2}{2+\alpha}};
\]
for instance, $0.15\leq A(3/4)\leq0.2132\ldots$. If $X=0$, i.e. $a_3=\tfrac{2-\alpha}{2(1-\alpha)}a_2^2$, then \eqref{eq:a2_improved} is the best bound on $|a_2|$ compatible with \eqref{eq:coeff_family} alone, so improving it requires information beyond \eqref{eq:coeff_family}.
\end{rem}

\begin{prop}\label{prop:fc}
For $0<c<1$ let $f_c(z)=\frac{2}{c}\bigl(1-\sqrt{1-cz}\bigr)=z+\frac c4z^2+\frac{c^2}8z^3+\cdots$ (principal branch). Then $f_c\in CC_{\alpha(c)}$, where
\[
\alpha(c)=\frac12\Bigl(1+\sqrt{\tfrac{1-c}{1+c}}\Bigr)=\inf_{x,z\in\D}\operatorname{Re}G_{f_c}(x,z),
\]
and the infimum is approached only as $(x,z)\to(1,-1)$.
\end{prop}

\begin{proof}
Let $p=\sqrt{1-cz}$ and $q=\sqrt{1-cx}$. Since $f_c'=1/p$, $f_c(z)-f_c(x)=\frac2c(q-p)$ and $c(z-x)=q^2-p^2$, a direct computation gives
\begin{equation}\label{eq:Gfc}
G_{f_c}(x,z)=\frac{1+pq}{p\,(p+q)} .
\end{equation}
As $\operatorname{Re}p,\operatorname{Re}q>0$ for $|z|,|x|<1/c$, the function $G_{f_c}$ is analytic in a neighbourhood of $\overline{\D}\times\overline{\D}$ (the diagonal is a removable singularity). Applying the minimum principle first in $z$ and then in $x$, the infimum of $\operatorname{Re}G_{f_c}$ over $\D\times\D$ equals its minimum over $\partial\D\times\partial\D$ and is not attained in $\D\times\D$.

Let $z=e^{i\theta}$, $x=e^{i\eta}$, $x\ne z$, and $r_z=|1-cz|$, $r_x=|1-cx|$. Then $(z+x)/(z-x)\in i\mathbb R$ and $q-p=c(z-x)/(p+q)$, so
\[
2\operatorname{Re}G_{f_c}-1=\operatorname{Re}\Bigl[\frac qp\cdot\frac{2z}{z-x}\Bigr]=\frac{\operatorname{Im}W}{r_z\sin\beta},
\qquad \beta=\tfrac{\theta-\eta}2,\quad W=q\,e^{-i\eta/2}\,\overline{p\,e^{-i\theta/2}} .
\]
Since $|W|^2=r_zr_x$ and $W^2=(\bar x-c)(z-c)$, the identity $(\operatorname{Im}W)^2=\frac12(|W|^2-\operatorname{Re}W^2)$ and $r_z^2=1+c^2-2c\cos\theta$ yield $(\operatorname{Im}W)^2=\frac14\bigl(|z-x|^2-(r_z-r_x)^2\bigr)$. As $2\operatorname{Re}G_{f_c}-1$ never vanishes on the connected set $(\partial\D\times\partial\D)\setminus\{x=z\}$ and is positive at $(1,-1)$, we obtain
\begin{equation}\label{eq:ReGfc}
\operatorname{Re}G_{f_c}(x,z)=\frac12+\frac1{2r_z}\sqrt{1-\Bigl(\frac{r_z-r_x}{|z-x|}\Bigr)^2}.
\end{equation}
Now $|r_z-r_x|\le c|z-x|$ and $r_z\le1+c$, hence
\[
\operatorname{Re}G_{f_c}(x,z)\ \ge\ \frac12+\frac{\sqrt{1-c^2}}{2(1+c)}=\alpha(c),
\]
with equality only for $z=-1$, $x=1$. Indeed, there $p=\sqrt{1+c}$, $q=\sqrt{1-c}$, $1+pq=\frac12(p+q)^2$, and \eqref{eq:Gfc} gives $G_{f_c}(1,-1)=\frac{p+q}{2p}=\alpha(c)$.
\end{proof}

\begin{rem}\label{rem:fc}
The map $c\mapsto\alpha(c)$ is a decreasing bijection of $(0,1)$ onto $(\frac12,1)$. Since $a_2=c/4$, one checks that $a_2>\frac{1-\alpha(c)}{1+\alpha(c)}$ if and only if $c<\frac45$; thus, for $\alpha\in(\frac23,1)$, the M\"obius maps of Example~\ref{ex:mobius} do not maximize $|a_2|$ in $CC_\alpha$. They are not extremal for $|\varphi|$ or $|f|$ either: for $c=3/5$, so that $\alpha=3/4$, the Schwarz function of Proposition~\ref{prop:repz0} is $\varphi(z)=2\bigl(1-\sqrt{1-cz}\bigr)$, and
\begin{gather*}
\sup_{\D}|\varphi|\geq\varphi(1)=2\bigl(1-\sqrt{2/5}\bigr)\approx0.735>\tfrac47=\tfrac1{1+\alpha},\\
\sup_{\D}|f_c|\geq f_c(1)=\tfrac{10}{3}\bigl(1-\sqrt{2/5}\bigr)\approx1.225>\tfrac76=\tfrac{1+\alpha}{2\alpha}.
\end{gather*}
What remains true in general is the bound $\operatorname{Re}\{f(z)/z\}>2^{\alpha-1}$, which follows from $CC_\alpha\subseteq S^*_{(1+\alpha)/2}$ since then $f(z)/z\prec(1-z)^{-(1-\alpha)}$. We do not know whether there is a constant $k(\alpha)<1$ such that $|\varphi|\leq k(\alpha)$ in $\D$ for every $f\in CC_\alpha$. Finally, $\inf_\D\operatorname{Re}\{1+zf_c''/f_c'\}=\frac{2+c}{2(1+c)}>\alpha(c)$, so that $f_c\in C_\gamma\setminus CC_\gamma$ for $\alpha(c)<\gamma\le\frac{2+c}{2(1+c)}$.
\end{rem}

\begin{prop}\label{prop:a3}
If $f(z)=z+a_2z^2+a_3z^3+\cdots \in CC_\alpha$, then
\begin{equation}\label{eq:bounda3_a2}
|a_3-a_2^2|\leq\frac{1-\alpha-|a_2|^2}{3}.
\end{equation}
In particular, using \eqref{eq:a2_improved},
\begin{equation}\label{eq:a3_bound}
  |a_3|\leq\frac{(1-\alpha)(2-\alpha)}{2+\alpha}.
\end{equation}

\end{prop}

\begin{proof}
Write $\varphi(x)=x\delta(x,0)$ as in Proposition \ref{prop:repz0} and set $d_0=\delta(0,0)=\varphi'(0)$,
$d_x=\delta_x(0,0)=\tfrac12\varphi''(0)$. Differentiating
$f(z)=z/(1-(1-\alpha)\varphi(z))$ three times at $z=0$ gives
\begin{equation}\label{eq:a2a3_from_phi}
  a_2=(1-\alpha)\varphi'(0)=(1-\alpha)d_0, \qquad
  a_3=(1-\alpha)\bigl[(1-\alpha)d_0^2+d_x\bigr].
\end{equation}
On the other hand, equating the coefficients of $z$ in the series expansions of $\varphi'(z)/(1-(1-\alpha)\varphi(z))$ and $\psi(z)/(z(1-\psi(z)))$ at $z=0$, with
$d_z=\delta_z(0,0)=\psi''(0)/2$ one obtains
\begin{equation}\label{eq:compat_first_order}
  d_z \;=\; 2d_x - \alpha\, d_0^2.
\end{equation}
Since $\delta:\D\times\D\to\D$ is analytic, its restriction to the diagonal $t\mapsto\delta(t,t)$ is an analytic self-map of $\D$ with value $d_0$ at $t=0$; the Schwarz-Pick lemma applied to this restriction gives
\begin{equation}\label{eq:diag_schwarz}
  |d_x+d_z|\leq 1-|d_0|^2.
\end{equation}
Substituting \eqref{eq:compat_first_order} into \eqref{eq:diag_schwarz} we obtain $|d_x+d_z|=|3d_x-\alpha d_0^2| \leq 1-|d_0|^2$, then by \eqref{eq:a2a3_from_phi} 
\[ 
  \left|\frac{3}{1-\alpha}(a_3-a_2^2)-\frac{\alpha a_2^2}{(1-\alpha)^2}\right|\leq 1-\left|\frac{a_2}{1-\alpha}\right|^2.
\] 
 Thus, \eqref{eq:bounda3_a2} is obtained using the triangle inequality. 
\end{proof}

\begin{rem}
For $\alpha=0$, bound \eqref{eq:a3_bound} gives the classical Robertson bound for convex functions \cite{Rob1936}: $|a_3|\leq 1$ for $f\in C$, with equality for $f(z)=z/(1-z)$.  Also, $\alpha=0$ in \eqref{eq:bounda3_a2} gives Trimble's inequality for $f\in C$ (see \cite{T1975}).
\end{rem}

\section{Subordination of the difference quotient and derivative}

The main result regarding subordination that we have found for $f\in CC_{\alpha}$ is the following:

\begin{theo}\label{theo:subord}
Let $f \in CC_\alpha$, $\alpha\in [0,1)$. Then
\begin{equation}\label{eq:diffquot_subord}
  \frac{x}{f(x)}\cdot\frac{f(z)-f(x)}{z-x}
  \prec\frac{1}{(1-z)^{1-\alpha}},
\end{equation}
where the subordination is in the variable $z$ for each fixed $x\in\D$.
\end{theo}

\begin{proof}
Since $f \in CC_\alpha$, by Herglotz's classical representation result, the function $G_f$ defined by \eqref{G} satisfies
\[
  \frac{G_f(x,z)-\alpha}{1-\alpha}=\int_0^{2\pi} \frac{1+e^{it}z}{1-e^{it}z}\,d\mu_x(t),
\]
where $\mu_x$ is a positive measure on $[0,2\pi]$ with $\int_0^{2\pi}d\mu_x(t)=1$. Whence
\begin{align*}
  \frac{2zf'(z)}{f(z)-f(x)}-\frac{z+x}{z-x}
  =\int_0^{2\pi}\frac{1+(1-2\alpha)e^{it}z}{1-e^{it}z}\,d\mu_x.
\end{align*}
A direct calculation gives
\begin{equation}\label{eq:G_herglotz}
  \frac{f'(z)}{f(z)-f(x)} - \frac{1}{z-x}
  \;=\; (1-\alpha)\int_0^{2\pi}\frac{e^{it}}{1-e^{it}z}\,d\mu_x(t).
\end{equation}
Substituting $z\mapsto rz$, $r\in[0,1]$, and
integrating \eqref{eq:G_herglotz} in $r$, we have
\[
  \int_0^1\!\left(\frac{f'(rz)}{f(rz)-f(x)}-\frac{1}{rz-x}\right)dr=(1-\alpha)\int_0^{2\pi}
  \int_0^1\frac{e^{it}}{1-e^{it}rz}\,dr\,d\mu_x(t).
\]
A straightforward calculation yields
 \[ \frac{1}{z}\log\left(\frac{x}{f(x)}\cdot\frac{f(z)-f(x)}{z-x}\right)= \frac{1-\alpha}{z}\int_0^{2\pi}\log\left(\frac{1}{1-e^{it}z}\right)\,d\mu_x(t),
\]
equivalently
\begin{equation}\label{eq:log_subord}
  \log\left(\frac{x}{f(x)}\cdot\frac{f(z)-f(x)}{z-x}\right)=\int_0^{2\pi}\log\left(\frac{1}{(1-e^{it}z)^{1-\alpha}}\right).
\end{equation}
The right-hand side of~\eqref{eq:log_subord} is an integral average of functions of the form $\log(1-e^{it}z)^{-(1-\alpha)}$, each of which is the logarithm of a function mapping $\D$ into a convex
domain. Since the image of $z\mapsto \log(1-z)^{-(1-\alpha)}$ is a convex sector and convexity is preserved under affine combinations, it follows that 
\[
  \log\left(\frac{x}{f(x)}\cdot\frac{f(z)-f(x)}{z-x}\right) \prec \log\left(\frac{1}{(1-z)^{1-\alpha}}\right).
\]
Exponentiating, which preserves subordination since the exponential is univalent on a suitable sector,  and using \eqref{eq:log_subord} we obtain \eqref{eq:diffquot_subord}.
\end{proof}

\begin{coro}\label{coro:zff}
If $f \in CC_\alpha$, $\alpha\in [0,1)$, then
\begin{equation}\label{eq:zff_subord}
  \frac{zf'(z)}{f(z)} \;\prec\; \frac{1}{(1-z)^{1-\alpha}}.
\end{equation}
\end{coro}

\begin{proof}
A classical argument shows that it suffices to make $x$ tend towards $z$ in  \eqref{eq:diffquot_subord}.
\end{proof}

\begin{coro}\label{coro:fbound}
Let $f \in CC_\alpha$, $\alpha \in [0,1)$, then
\begin{equation*}\label{eq:fprime_bound}
  \frac{1}{(1+(1-\alpha)|z|)(1+|z|)^{1-\alpha}}\leq|f'(z)| \leq \frac{1}{(1-(1-\alpha)|z|)(1-|z|)^{1-\alpha}},  \qquad z \in \D.
\end{equation*}
\end{coro}

\begin{proof}
By \eqref{eq:zff_subord} 
\[
 \frac{1}{(1+|z|)^{1-\alpha}}\frac{|f(z)|}{|z|} \leq|f'(z)|\leq\frac{1}{(1-|z|)^{1-\alpha}}     \frac{|f(z)|}{|z|}.
\]
Now, by \eqref{eq:repz0} and the Schwarz lemma
\begin{equation}\label{prop:repz1}
 \frac{1}{1+(1-\alpha)|z|}\leq \left|\frac{f(z)}{z}\right|\leq \frac{1}{1-(1-\alpha)|z|},
\end{equation}
for all $z\in \D$.
\end{proof}

It is important to note that the bound obtained in \eqref{prop:repz1} improves upon the one obtained in Lemma 2.5 in \cite{paper}.


\end{document}